\documentclass[a4paper,10pt]{amsart}

\usepackage[english]{babel}
\usepackage{dsfont} 

\usepackage{amssymb} 
\usepackage{amsthm} 

\newtheorem{thm}{Theorem}

\newtheorem{lem}[thm]{Lemma}

\newtheorem{exa}[thm]{Example}

\newcommand{\C}{\mathds{C}}
\newcommand{\R}{\mathds{R}}
\newcommand{\N}{\mathds{N}}

\newcommand{\cE}{\mathcal{E}}

\newcommand{\cH}{\mathcal{H}}

\newcommand{\cD}{\mathcal{D}}
\newcommand{\cK}{\mathcal{K}}

\newcommand{\ux}{\underline{x}}
\newcommand{\ii}{\mathrm{i}}

\newcommand{\ov}{\overline}
\newcommand{\uy}{\underline{y}}

\author{Konrad Schm\"udgen}
\address{Universit\"at Leipzig, Mathematisches Institut, Augustusplatz 10/11, D-04109 Leipzig, Germany}
\email{schmuedgen@math.uni-leipzig.de}

\begin{title}
{On  the multidimensional  $\cK$-moment problem}
\end{title}
\date{}
\begin{document}

\maketitle

\begin{abstract}
We prove a solvability theorem for the Stieltjes problem on $\R^d$ which is based on the multivariate Stieltjes condition 
$\sum_{n=1}^\infty L(x_j^{n})^{-1/(2n)} {=}+\infty$, $j=1,\dots,d.$ 
This result is applied to derive a new solvability theorem for the moment problem on unbounded semi-algebraic subsets of $\R^d$.
\end{abstract}

\textbf{Key  words:} moment problem, Stieltjes vector,  semi-algebraic set

\maketitle

\section {Introduction}

It is well known that the moment problem on unbounded semi-algebraic sets of $\R^d$ leads to new and principal difficulties, concerning the existence and the uniqueness of solutions. A rather general existence theorem is the fibre theorem (\cite{sch2015}, see also \cite[Chapter 4]{marshall1} or \cite[Chapter 13]{sch2017}). Most of the general existence results that are based on positivity conditions 
can be derived from this theorem. The fibre theorem requires the existence of 
``sufficiently many" bounded polynomials on the semi-algebraic set.

 However, for many  unbounded sets there are no bounded polynomials except from the constants and  positivity on the corresponding preorderings is not enough to ensure the existence of solutions. For such sets additional assumptions such as growth conditions on the moments  are required.
The most famous
 and powerful result in this respect is  Nussbaum's theorem \cite{nussbaum1}. It says that if $L$ is a linear functional on $\R[x_1,\dots,x_d]$ such that $L(p^2)\geq 0$ for $p\in \R[x_1,\dots,x_d]$ and satisfying the {\it Carleman condition}
\begin{align}\label{carle}
\sum_{n=1}^\infty L(x_j^{2n})^{-1/(2n)} =+ \infty,~~~j=1,\dots,d,
\end{align}
then $L$ is a moment functional with a unique representing measure on $\R^d$. 

The aim of this paper is to prove two theorems on the multidimensional moment problem  that are based on another growth condition (\ref{ass2}), called {\it Stieltjes condition}.  Comparing the powers of $L(x_j^{2n})$ in both cases we see that the Stieltjes condition (\ref{ass2}) is weaker than the Carleman condition  (\ref{carle}). The first main result (Theorem \ref{stielt1}) is the counter-part of Nussbaum's theorem for the Stieltjes moment problem.
The second main result (Theorem \ref{stielt2}) deals with the $\cK$-moment problem for some in general noncompact semi-algebraic set $\cK$.

Let $\cK$ be  a closed subset of $\R^d$. The $\cK$-moment problem asks the following question: {\it When is a real $d$-sequence $s=(s_n)_{n\in \N_0^d}$ a  $\cK$-moment sequence}, that is, when does there exist a Radon measure $\mu$ supported on $\cK$ such that 
\begin{align}\label{moments}
s_n=\int_\cK~ x^n\, d\mu(x)~~ \textrm{for all} ~~n\in \N_0^d?
\end{align}
(Here we tacitly mean that the corresponding integral is finite and equal to $s_n$ and we use the multi-index notation $x^n=x_1^{n_1}\cdots x_d^{n_d}$ for $n=(n_1,\dots,n_d)\in \N_0^d$.)

There is a one-to-one correspondence between real $d$-sequences $s=(s_n)_{n\in \N_0^d}$ and real-valued  linear functionals on the polynomial algebra $\R[x_1,\dots,x_d]$ given by $L_s(x^n)=s_n, n\in \N_0^d$. The functional $L_s$  is  called the Riesz functional  associated with $s$. Then (\ref{moments}) is equivalent to
\begin{align}\label{momentf}
L_s(p)=\int_\cK~ p(x)\, d\mu(x)~~~ \textrm{for all} ~~p\in \R[x_1,\dots,x_d].
\end{align}
A linear functional on $\R[x_1,\dots,x_d]$ of this form is called a {\it $\cK$-moment functional}. In the special case $\cK=\R_+^d$, where  $\R_+:=[0,+\infty)$, we call such a functional {\it Stieltjes moment functional}. If there is no support requirement for the measure, that is, if $\cK=\R^d$, the functional is called a {\it moment functional}. The functional-analytic approach to the moment problem uses the integral representation (\ref{momentf}) of functionals rather than the representation (\ref{moments}) of moments.
 
 \section{A theorem for the Stieltjes moment problem on $\R^d$}
 
 In the case $d=1$,
 a  necessary and sufficient condition for being a Stieltjes moment sequence is that the sequences $(s_n)_{n\in \N_0}$ and $(s_{n+1})_{n\in \N_0}$ are positive semi-definite, or equivalently, 
 \begin{align}\label{stid=1}
 L_s(p^2)\geq 0~~\textrm{and}~~ L_s(xp^2)\geq 0~~ \textrm{for}~ p\in \R[x].
 \end{align} In dimensions $d\geq 2$, the Stieltjes moment problem is much more subtle and there is no such simple solvability criterion. The counter-part of (\ref{stid=1}) are the conditions in (\ref{ass1}). However, (\ref{ass1}) is necessary, but it is not sufficient for being a Stieltjes moment functional. The reason for the latter is that there are nonnegative polynomials on $\R_+$ which are not of the form form $p_1+xp_2$, with $p_1, p_2$ sums of squares. 
 
 The following theorem says that the positivity assumption (\ref{ass1}) {\it and}  the Stieltjes condition (\ref{ass2}) imply the existence and  uniqueness of a solution of the Stieltjes  moment problem.

\begin{thm}\label{stielt1}
Let $L$ be a linear functional on $\R[x_1,\dots,x_d]$ such that 
\begin{align}\label{ass1}
L(p^2)\geq 0~~~ \textit{and}~~~ L(x_jp^2)\geq 0 ~~\textit{for}~~ p\in\R[x_1,\dots,x_d],~ j=1,\dots,d.
\end{align}
Suppose that
\begin{align}\label{ass2}
\sum_{n=1}^\infty L(x_j^{n})^{-1/(2n)} =+ \infty~~~ \textrm{for}~~j=1,\dots,d.
\end{align}
Then $L$ is Stieltjes moment functional on $\R^d$ which has a unique representing measure.
\end{thm}

Before we discuss this theorem we state a well known technical lemma. For the convenience of the reader we include a proof of the lemma.
\begin{lem}\label{techlem}
Let $s=(s_n)_{n\in \N_0}$ be a real sequence with $s_0=1$. Suppose that $s$ and the shifted sequence  $Es:=(s_{n+1})_{n\in \N_0}$ are positive semidefinite, that is, 
\begin{align*}
\sum_{k,l=0}^n s_{k+l}\xi_k\,\xi_l \geq 0~~ \textrm{and}~~ \sum_{k,l=0}^n s_{k+l+1}\xi_k\,\xi_l \geq 0~~~\textrm{for all} ~~ \xi_0,\xi_1,\dots,\xi_n\in \R, n\in \N_0.
\end{align*}
Let $m\in \N$. Then 	we have	\begin{align}\label{ass5}
\sum_{n=1}^\infty s_{n}^{-1/(2n)}=+ \infty
\end{align} if and only if
\begin{align}\label{ass6}
\sum_{n=1}^\infty s_{nm}^{-1/(2nm)} =+ \infty.
\end{align}	 	
\end{lem}
\begin{proof} Clearly, all summands are nonnegative. Hence, since $(s_{nm}^{-1/(2nm)})_{n\in \N}$ is a subsequence of  $(s_n^{-1/(2n)})_{n\in \N}$, it is trivial that (\ref{ass6}) implies (\ref{ass5}).

We prove the converse. Because the sequences $s$ and $Es$ are positive semidefinite, Stieltjes theorem (see e.g. \cite[Theorem 3.12]{sch2017}) implies that $s$ is a Stieltjes moment sequence, that is, there is a Radon measure on $\R_+$ such that $s_n=\int_0^\infty x^n d\mu(x)$ for $n\in \N_0$. Clearly, $\mu(\R_+)=s_0=1$. Suppose $0<k<l$. The H\"older inequality, applied with $p=\frac{l}{k}, q=\frac{l}{l-k}$, yields
\begin{align*}
s_k=\int_0^\infty x^k \cdot 1\, d\mu\leq \Big(\int_0^\infty (x^k)^p d\mu\Big)^{1/p} \Big(\int_0^\infty 1\, d\mu\Big)^{1/q} =\Big(\int_0^\infty x^l\, d\mu\Big)^{k/l} s_0^{1/q}= s_l^{k/l},
\end{align*}
so that $s_k^{1/k}\leq s_l^{1/l}$. Hence the sequence $(s_k^{-1/k})_{k\in \N}$ is nonincreasing and therefore
\begin{align*}
\sum_{n=m}^\infty  s_n^{-1/(2n)}=\sum_{l=0}^{m-1}\sum_{n=1}^\infty s_{nm+l}^{-1/(2(nm+l))}\leq m \sum_{n=1}^\infty  s_{nm}^{-1/(2nm)}.
\end{align*}
Thus, if the sum in (\ref{ass6}) is finite, so the sum in (\ref{ass5}), that is,  (\ref{ass5}) implies (\ref{ass6}).
\end{proof}

Now suppose $L$ is as in Theorem \ref{stielt1}. If $L(1)=0$, it follows from the Cauchy-Schwarz inequality that $L=0 $, so the assertion is trivial. Otherwise,  $L(1)>0$. Then, setting $L'=L(1)^{-1}L$, $L'$ satisfies the assumptions (\ref{ass1}) and (\ref{ass2}) and in addition $L'(1)=1$.  (Condition  (\ref{ass2}) for $L'$  follows  from the fact that $\lim_n L(1)^{1/(2n)}=1$).) 

Thus we can assume without loss of generality in Theorem \ref{stielt1} that $L(1)=1$. Then, for $j=1,\dots,d$, the sequence $(s_n:=L(x_j^n))_{n\in \N_0}$ satisfies the assumptions of Lemma \ref{techlem}. Therefore,  by this lemma condition (\ref{ass2}) is equivalent to
\begin{align}\label{stieltjm}
\sum_{n=1}^\infty L(x_j^{mn})^{-1/(2nm)} =+ \infty ~~~\textrm{for one (then for all)}~~ m\in \N.
\end{align}
The case $m=4$ is used in the proof of Theorem \ref{stielt1} in Section \ref{section2} below.

The proof of Theorem \ref{stielt1} given in Section \ref{section2} is essentially based on the operator-theoretic approach to the  moment problem, see \cite[Section 12.5]{sch2017} or \cite{vasilescu}.
While Carleman's condition  in  Nussbaum's theorem leads to quasi-analytic vectors \cite{nussbaum1}, Theorem \ref{stielt1} uses  Stieltjes vectors, see  (\ref{defstv}) for the definition of a Stieltjes vector. The crucial operator-theoretic result is the following: If a nonnegative symmetric operator  has a dense set of Stieltjes vectors, then it is essentially self-adjoint. This theorem was first proved by Nussbaum \cite{nussbaum2} and a  later independently by Masson and McClary \cite{masson}, see also \cite[Theorem 7.15]{sch2012}. 

The Carleman condition goes back to T. Carleman's classical work \cite{carleman}. It is now a standard tool in the theory of the classical moment problem  and used in many papers such as \cite{berg}, \cite{nussbaum1}, \cite{stochel}, \cite{marshall2}, \cite{jeu}, \cite{lass}. The multidimensional classical moment problem on $\R^d$ and its interplay with real algebraic geometry are treated in \cite{sch2017} and \cite{marshall1}.
\section{A  theorem on the $\cK$-moment problem for semi-algebraic sets}

Theorem \ref{stielt1} is essentially used to derive the following result concerning the $\cK$-moment problem for  semi-algebraic sets.

\begin{thm}\label{stielt2}
Let\, ${\sf f}=\{f_1,\dots,f_m\}$ be the finite set of polynomials $f_j\in \R[x_1,\dots,x_d]$, $d\in \N$, which generate the polynomial algebra $\R[x_1,\dots,x_d]$. Suppose that $L$ is a linear functional on $\R[x_1,\dots,x_d]$ satisfying 
\begin{align}\label{ass3}
L(p^2)\geq 0~~~ \textit{and}~~~ L(f_jp^2)\geq 0~~\textit{for}~~ p\in \R[x_1,\dots,x_d],\, j=1,\dots,m,
\end{align}
and
\begin{align}\label{ass4}
\sum_{n=1}^\infty L(f_j^{n})^{-1/(2n)} =+ \infty~~~ \textrm{for}~~ j=1,\dots,m.
\end{align}
Then $L$ is  moment functional. It has a unique representing measure. This measure is supported on the semi-algebraic set 
\begin{align}\label{defkf}
\cK({\sf f}):=\{ x\in \R^d: f_1(x)\geq 0,\dots, f_k(x)\geq 0\}.
\end{align}
\end{thm}

\smallskip

In Theorem  \ref{stielt2} we assumed that the polynomials $f_1,\dots,f_m$ generate the whole polynomial algebra $\R[x_1,\dots,x_d]$.
Without this assumption the proof of   Theorem  \ref{stielt2}  in Section \ref{section2} remains valid and shows that
$L(q)=\int_{\cK({\sf f})} q(\lambda)\, d\mu(\lambda)$ for all polynomials $q$ in the subalgebra generated by $f_1,\dots,f_m$.

The power and usefulness	of Theorem \ref{stielt2} is nicely illustrated by the following example.
	\begin{exa}
Suppose that $d=2$ and $m=2$. We fix a number $k\in \N$ and define $f_1(x)=x_2-x_1^k ,f_2(x)=x_1$. Obviously, the polynomials $f_1, f_2$ generate the polynomial algebra $\R[x_1,x_2]$. The corresponding semi-algebra set $\cK({\sf f})$ is the part of the plane above the curve $x_2=x_1^k$ lying in the positive octant. In the case $k=1$ we get the cone in the positive octant bounded by the $x_2$-axis and the line $x_2=x_1$.

In this case, assumptions (\ref{ass1}) and (\ref{ass2}) take the following  form: 
\begin{align}\label{as1}
L(p^2)& \geq 0,~ L((x_2-x_1^k)p^2)\geq 0, ~ L(x_1p^2)\geq 0 ~~~\textrm{for all}~~ p\in \R[x_1,x_2],\\																									
\sum_{n=1}^\infty~ &L((x_2-x_1^k)^n)^{-1/(2n)}	=+\infty ,~~~		\sum_{n=1}^\infty L(x_1^n)^{-1/(2n)}	=+\infty. \label{as2} 		
\end{align}																																																																						Then Theorem \ref{stielt2} states that each  linear functional $L$ on $	\R[x_1,x_2]$ satisfying (\ref{as1}) and (\ref{as2}) is a moment functional with a unique representing measure and this measure is supported on $\cK({\sf f})$.

Let us try  to treat the moment problem for this semi-algebraic set $\cK({\sf f})$ by other methods. It is easily verified that the constant polynomials are the only polynomials that are bounded on the set  $\cK({\sf f})$. Hence the fibre theorem \cite{sch2015} does not apply to the set $\cK({\sf f})$. 

Another possibility is to assume in addition to (\ref{as1}) the multivariate Carleman condition (\ref{carle}), that is,
\begin{align}\label{as3}																									
\sum_{n=1}^\infty~ &L(x_1^{2n})^{-1/(2n)}	=+\infty ~~\textrm{and}~~		\sum_{n=1}^\infty L(x_2^{2n})^{-1/(2n)}	=+\infty. 		
\end{align}	
Then $L$ is a moment functional by Nussbaum's theorem with a unique representing measure and by Lasserre's localization theorem (\cite{lass}, see also \cite[Theorem 14.25]{sch2017}) this measure is supported on  $\cK({\sf f})$. 

As already noted above, the Stieltjes condition (\ref{ass2}) is weaker than the Carleman condition (\ref{carle}). Comparing the two approaches mentioned in the preceding paragraphs shows another advantage of Theorem \ref{stielt2} over Nussbaum's theorem.  Assumption (\ref{as2}) of Theorem \ref{stielt2} requires growth conditions of the functional $L$ at the boundary curves $x_2=x_1^k$ and $x_1=0$ of the  set $\cK({\sf f})$, while the Carleman assumptions (\ref{as3}) are  growth conditions of $L$ at the coordinate axis.	That is, Theorem \ref{stielt2} contains weaker growth assumptions concerning the powers of $L(x_j^n)$ and it is better adapated to the geometric form of the semi-algebraic set  $\cK({\sf f})$.
\end{exa}								
																	 																							 						
\section{Proofs of Theorems	\ref{stielt1} and \ref{stielt2} }\label{section2}				

Throughout this proof we abbreviate the polynomial algebras $\R[y_1,\dots,y_m]$ by $\R_m[\uy]$,  $\C[y_1,\dots,y_m]$ by $\C_m[\uy]$, and $\R[x_1,\dots,x_d]$ by $\R_d[\ux]$.

As noted in Section \ref{section2}, we can assume without loss of generality that $L(1)=1$.
We extend the $\R$-linear functional $L$ on  $\R_d[\ux]$ to a $\C$-linear functional  $L_\C$ on  $\C_d[\ux]$ by setting $$L_\C(p_1+\ii p_2)=L(p_1)+\ii L(p_2),\quad p_1,p_2\in \R_d[\ux].$$ Note that $\C_d[\ux]$ is a complex  $*$-algebra with involution $(p_1+\ii p_2)^+:=p_1-\ii p_2$ for $ p_1,p_2\in \R_d[\ux]$. Since $(p_1+\ii p_2)^+(p_1+\ii p_2)=p_1^2+p_2^2$, the first condition of assumption (\ref{ass1}) implies that $L_\C(p^+p)\geq 0$ for $p\in\C_d[\ux]$, that is, $L_\C$ is a positive linear functional on the complex unital $*$-algebra $\C_d[\ux]$. Hence the GNS construction 
 applies to this functional $L_\C$.

We briefly recall the corresponding GNS representation (see   e.g.
\cite[Section 3.5]{sch2020} for a detailed exposition). To simplify the notation let us suppress the index $\C$ and write $L$ instead of $L_\C$.
The GNS representation $\pi_L$ of $L$ acts on a complex inner product space $(\cD_L,\langle \cdot, \cdot\rangle) $. Let $\cH_L$ denote the Hilbert space completion of $\cD_L$. The representation $\pi_L$ has an algebraically cyclic vector $\varphi_L$, that is, $\cD_L=\pi_L(\C_d[\ux])\varphi_L$,	and it satisfies 
\begin{align}\label{posLvar}
L(p)=\langle \pi_L(p)\varphi_L,\varphi_L\rangle,~~p\in \C_d[\ux].
\end{align}	
For $p,q\in \C_d[\ux]$ and $\eta, \zeta\in \cD_L$,  we have	$\pi_L(p)\pi_L(q)\eta=\pi_L(pq)\eta$, $\pi_L(1)\eta=\eta$, and
\begin{align*}
\langle \pi_L(p)\eta,\zeta \rangle =\langle \eta,\pi_L(p^+)\zeta\rangle.
\end{align*}																						
In the subsequent proof we will freely use these properties.

Our proof is essentially based on the concept of a Stieltjes vector.
Let $T$ be a symmetric operator on a Hilbert space $\cH$. A vector $\varphi$ of $\cH$ is called a {\it Stieltjes vector} for $T$ (\cite{nussbaum2}, see also \cite[Definition 7.1]{sch2012}) if	$\varphi\in \cD(T^n)$ for all $n\in \N$ and
\begin{align}\label{defstv}
\sum_{n=1}^\infty \|T^n\varphi \|^{-1/2n}=+ \infty.
\end{align}

The crucial step of the proof of Theorem \ref{stielt1} is the next lemma. It shows that assumption (\ref{ass2}) implies that the domain $\cD_L$ contains enough  Stieltjes vectors for the representation operators $\pi_L(x_j)$.
\begin{lem}\label{stielcom} Suppose $S$ is a symmetric  operator on $\cD_L$ which leaves $\cD_L$ invariant and commutes with all operators $\pi_L(\R_d[\ux])$. 
Then $S\varphi_L$ is a Stieltjes vector for $\pi_L(x_j), j=1,\dots,d$. In particular, $\varphi_L$ is a Stieltjes vector for each operator $\pi_L(x_j)$.
\end{lem} 
\begin{proof}  Let $j=1,\dots,d$ and $n\in \N$. Using that the assumptions on the symmetric operator $S$ and basic properties of the GNS representation $\pi_L$ we derive 
\begin{align*}
\|\pi_L(x_j)^n S\varphi_L\|^4&=\langle \pi_L(x_j)^n S\varphi_L,\pi_L(x_j)^n S\varphi_L\rangle^2\\ &=\langle \pi_L(x_j)^{2n} S\varphi_L, S\varphi_L\rangle^2=\langle  \pi_L(x_j)^{2n} \varphi_L, S^2\varphi_L\rangle^2\\ &\leq  \|\pi_L(x_j)^{2n} \varphi_L\|^2\, \| S^2\varphi_L\|^2=\langle\pi_L(x_j)^{2n} \varphi_L,\pi_L(x_j)^{2n}\varphi_L\rangle\, \| S^2\varphi_L\|^2\\ 
&=\langle\pi_L(x_j^{4n}) \varphi_L,\varphi_L\rangle\, \| S^2\varphi_L\|^2\leq L(x_j^{4n})\, (1+\| S^2\varphi_L\|)^2
\end{align*}
and hence
\begin{align}\label{in1}
\|\pi_L(y_j)^n\varphi\|^{1/(2n)} \leq L(x_j^{4n})^{1/(8n)}\,(1+\| S^2\varphi_L\|)^{1/(8n)}.
\end{align}
Since $\lim_{n\to \infty} (1+\| S^2\varphi_L\|)^{1/(8n)}=1$, there is a constant $c>0$ such that 
\begin{align}\label{in2}
(1+\| S^2\varphi_L\|)^{1/(8n)}\leq c~~~\textrm{for all}~~ n\in \N.
\end{align} for all $n\in \N$.
Combining (\ref{in1}) and  (\ref{in2}) yields
\begin{align*}
\sum_{n=1}^\infty \|\pi_L(x_j)^n\varphi\|^{-1/(2n)} \geq c^{-1} \sum_{n=1}^\infty L(x_j^{4n})^{-1/(8n)}.
\end{align*}
Since we have assumed that $L(1)=1$, 
the discussion after Lemma \ref{techlem} in Section \ref{section2} applies and shows  assumption (\ref{ass2})  implies (\ref{stieltjm}). Setting $m=4$ in  (\ref{stieltjm}) it follows that the right hand side of the preceding inequality is infinite, so is the left hand side. By (\ref{defstv}) this shows that $S\varphi_L$ is a Stieltjes vector for $\pi_L(x_j)$. Setting $S=I$, we conclude that $\varphi_L$ is a Stieltjes vector $\pi_L(x_j)$. 
\end{proof}
\begin{lem}
The operators $T_j:=\ov{\pi_L(x_j)}$, $j=1,\dots,d$, are pairwise strongly commuting self-adjoint operators. 
\end{lem}
\begin{proof}
First we verify that $\pi_L(x_j)\geq 0$. For let $p\in \C_m[\uy]$. Then we can write $p=p_1+\ii p_2$ with $p_1,p_2\in \R_d[\ux]$ and $p^+p=p_1^2+p_2^2$. Therefore,  
 $$\langle \pi_L(x_j)\pi(p)\varphi_L,\pi(p)\varphi_L\rangle = \langle \pi_L(x_j(p_1^2+p_2^2)) \varphi_L,\varphi_L\rangle=L(x_j (p_1^2+p_2^2)\geq 0$$ by assumption (\ref{ass1}). Thus $\pi_L(x_j)\geq 0$. Therefore, since each vector of $\cD_L$ is a Stieltjes vector for 
 $\pi_L(x_j)$ by Lemma \ref{stielcom}, it follows  from Nussbaum's theorem in \cite{nussbaum2} (see also \cite[Theorem 7.15]{sch2012}) 
  that $\pi_L(x_j)$ is essentially self-adjoint. This means that its closure $T_j$ of the operator $\pi_L(x_j)$ is self-adjoint. Since $\pi_L(x_j)\geq 0$, $T_j$ is positive.
 Hence $\cE_j:=(T_j+I)\cD_L=\pi_L(x_j+1)\cD_L$ is dense in $\cH_L$ for $j=1,\dots,x$ by \cite[Proposition 3.15]{sch2012}.

We prove that the self-adjoint operators $T_j$ and $T_k$ strongly commute. 

Clearly, the symmetric operator $\pi_L(x_k+1)$ commutes with all operators of $\pi_L(\R_d[\ux])$. Therefore, Lemma \ref{stielcom} applies with  $S:=\pi_L(x_k+1)$ and shows that all  vectors of $\cE_k=\pi_L(x_k+1)\cD_L$ are Stieltjes vectors for the symmetric operator $\pi_L(x_j)$. Since $\cE_k$ is dense as noted above, $T_j\lceil \cE_k=\pi_L(x_j)\lceil\cE_k$ is essentially self-adjoint, again by \cite[Theorem 7.15]{sch2012}. Hence, by  \cite[Proposition 3.15]{sch2012}, $\cE_{jk}:=(T_j+I)\cE_k=(T_j+I)(T_k+I)\cD_L$ is dense in $\cH_L$. Let $\psi\in \cE_{jk}$. By the definition of $\cE_{jk}$  there is a vector $\varphi\in \cD_L$ such that $\psi=(T_j+I)(T_k+I)\varphi= (T_k+I)(T_j+I)\varphi$. Then $(T_k+I)^{-1}(T_j+I)^{-1}\psi=\varphi= (T_j+I)^{-1}(T_k+I)^{-1}\psi$. This shows that the operators $(T_j+I)^{-1}$ and $(T_k+I)^{-1}$ commute on the dense linear subspace $\cE_{jk}$. Since $T_j$ and $T_k$ are positive self-adjoint operators, these operators are bounded, so they commute everywhere on $\cH_L$. That is, the  resolvent operators $(T_j+I)^{-1}$ of $T_j$ and $(T_k+I)^{-1}$ of $T_k$ commute. Hence the self-adjoint operators $T_j$ and $T_k$ commute strongly by  \cite[Theorem 5.27]{sch2012}. 
\end{proof}

Thus, $\{T_1,\dots,T_d\}$ is a $d$-tuple of strongly commuting positive self-adjoint operators acting on the Hilbert space $\cH_L$. Therefore, by the multidimensional spectral theorem \cite[Theorem 5.23]{sch2012}, there exists a unique spectral measure $E$ on the Borel $\sigma$-algebra of $\R^d$ such that $T_j=\int \lambda_j\, dE(\lambda)$, $j=1,\dots,d$. 

Now we  proceed as in the standard operator approach to the moment problem (see e.g. \cite[p. 305]{sch2017}). We define a Radon measure $\mu$ on $\R^d$ by $\mu(\cdot):=\langle E(\cdot)1,1\rangle$. The vectors of the domain  $\cD_L$, in particular the polynomial $1$, are contained in the domains of all operator polynomials $p(T_1,\dots,T_d)$ for $p\in\C_d[\ux]$ and we have $\pi_L(p)\subseteq p(T_1,\dots,T_d)$.  Therefore, using the functional calculus and equation (\ref{posLvar}) we derive
\begin{align*}
\int p(\lambda)\, d\mu(\lambda)= \int p(\lambda)\, d\langle E(\lambda)1,1\rangle= \langle p(T_1,\dots,T_d)1,1\rangle = \langle \pi_L(p)\varphi_L,\varphi_L\rangle =L(p).
\end{align*}
That is, $L$ is a moment functional. Because the self-adjoint operators $T_j$ are positive, the spectral measure $E$, hence the measure $\mu$, is supported on $\R_+^d$, that is, $L$ is a Stieltjes moment functional.	Since all operators $\pi_L(x_j)$, $j=1,\dots,d,$ are essentially self-adjoint, $\mu$ is determinate by \cite[Theorem 14.2]{sch2017}, that is, $\mu$ is the unique representing measure for $L$. This completes the proof of Theorem \ref{stielt1}.

Now we turn to the proof of Theorem  \ref{stielt2}. We define a unital algebra homomorphism $\theta: \R_m[\uy]\mapsto \R_d[\ux]$ by $\theta(y_j)=f_j, j=1,\dots,m,$  and a linear functional $\tilde{L}$ on $\R_m[\uy]$ by $\tilde{L}(p):= L(\theta(p))$, that is, ${\tilde L}(p(y_1,\dots,y_m))= L(p(f_1,\dots,f_m))$, for $p\in \R_m[\uy]$. The assumptions (\ref{ass3}) and (\ref{ass4}) on  $L$ imply that the functional $\tilde{L}$ satisfies the assumptions (\ref{ass1}) and (\ref{ass2}). Therefore, by Theorem \ref{stielt1}, $\tilde{L}$ is a Stieltjes moment functional on $\R_m[\uy]$, that is, there exists a Radon measure $\nu$ on $\R_+^m$ such that 
\begin{align}\label{nurep}
\tilde{L}(p)=\int_{\R_+^m} p(t)\, d\nu(t)~~~\textrm{for}~~p\in \R_m[\uy].
\end{align}
Let $\tau:\cK({\sf f})\mapsto \R_+^m$ denote the mapping defined by $\tau(\lambda)=(f_1(\lambda),\dots,f_m(\lambda)).$ That $\tau$ maps 
$\cK({\sf f})$ into $\R_+^m$ follows from  the definition (\ref{defkf}) of the semi-algebraic set $\cK({\sf f})$. Let $\mu$ denote the pull-back of the measure $\nu$  with respect to the mapping $\tau$. Using (\ref{nurep}) we derive for $p\in \R_m[\uy]$,
\begin{align*}
L(p(f_1,& \dots,f_m))\equiv L(\theta(p))=\tilde{L}(p)=\int_{\R_+^m} p(y_1,\dots,y_m)\, d\nu(y)\\ &=\int_{\cK({\sf f})} p(\tau(\lambda))\, d\mu(\lambda) = \int_{\cK({\sf f})} p(f_1(\lambda),\dots,f_m(\lambda))\, d\mu(\lambda).
\end{align*}
By assumption the polynomials $f_1,\dots,f_m$ generate the  algebra $\R_d[\ux]$. Therefore, $\theta$ is surjective and the preceding equality implies that $L(q)=\int_{\cK({\sf f})} q(\lambda)\, d\mu(\lambda)$ for all (!) polynomials $q\in \R_d[\ux]$. By construction, the  measure $\mu$ is supported on the semi-algebraic set $\cK({\sf f})$. It has finite moments, because $L(q^2)=\int q(\lambda)^2\ d\mu <+\infty$ for $q\in \R[\ux]$. Thus $L$ is moment functional. The uniqueness of the representing measure $\mu$ of $L$  follows at once from the uniqueness of the representing measure $\nu$ of $\tilde{L}$ by Theorem \ref{stielt1}. Now the proof of Theorem \ref{stielt2} is complete.

\bibliographystyle{amsalpha}

\end{document}